\documentclass[11pt,a4paper]{article}
\usepackage{mathrsfs}
\usepackage{amsmath}
\usepackage{amsfonts}
\usepackage{amssymb}
\usepackage{graphicx}
\usepackage{enumerate}
\usepackage{arydshln}

\usepackage{makeidx,epsf,psfrag,,epsfig,graphics}%,amsmath}
\usepackage{amsmath,amssymb}
\usepackage{graphicx,subfigure,epsfig}
\usepackage{tikz}
\usepackage[ruled]{algorithm2e}
\newcommand*{\circled}[1]{\lower.7ex\hbox{\tikz\draw (0pt, 0pt)%
    circle (.5em) node {\makebox[1em][c]{\small #1}};}}

\begin{document}

\renewcommand{\thefootnote}{\fnsymbol{footnote}}

\newtheorem{theorem}{Theorem}[section]
\newtheorem{corollary}[theorem]{Corollary}
\newtheorem{definition}[theorem]{Definition}
\newtheorem{conjecture}[theorem]{Conjecture}
\newtheorem{question}[theorem]{Question}
\newtheorem{lemma}[theorem]{Lemma}
\newtheorem{proposition}[theorem]{Proposition}
\newtheorem{example}[theorem]{Example}
\newenvironment{proof}{\noindent {\bf Proof.}}{\hfill\rule{2mm}{2mm}\par\medskip}

\newcommand{\ec}{{\rm ecc}}

\setcounter{page}{1}
\pagestyle{empty}
\pagestyle{headings}

\title{The Steiner $k$-eccentricity on trees
\thanks{This work was supported by the National Natural Science Foundation of China (11861019), Guizhou Talent Development Project in Science and Technology (KY[2018]046), Natural Science Foundation of Guizhou ([2019]1047, [2020]1Z001), Foundation of Guizhou University of Finance and Economics(2019XJC04). Sandi Klav\v zar acknowledges the financial support from the Slovenian Research Agency (research core funding No.\ P1-0297 and projects J1-9109, J1-1693, N1-0095, N1-0108. 
}}
% $^\dag$ Corresponding author.

%\titlerunning{Steiner $3$-eccentricity on Trees}

\author{Xingfu Li$^a$, Guihai Yu$^a$, Sandi Klav\v{z}ar$^{b,c,d}$, Jie Hu$^a$, Bo Li$^a$\\
{\small  $^a$ College of Big Data Statistics, Guizhou University of Finance and Economics}\\
 {\small Guiyang, Guizhou, 550025, China}\\
{\small  {\tt xingfulisdu@qq.com; yuguihai@mail.gufe.edu.cn}}\\
{\small  {\tt jason.houu@gmail.com; hn.libo@163.com}}\\
{\small  $^b$ Faculty of Mathematics and Physics, University of Ljubljana, Slovenia}\\
{\small $^{c}$ Institute of Mathematics, Physics and Mechanics, Ljubljana, Slovenia} \\
{\small $^{d}$ Faculty of Natural Sciences and Mathematics, University of Maribor, Slovenia}\\
{\small {\tt sandi.klavzar@fmf.uni-lj.si}}
 }

\maketitle

\begin{abstract}
    We study  the Steiner $k$-eccentricity on trees, which generalizes the previous one in the paper [X.~Li, G.~Yu, S.~Klav\v{z}ar, On the average Steiner 3-eccentricity of trees, arXiv:2005.10319, 2020].  To support the algorithm, we achieve much stronger properties for the Steiner $k$-ecc tree than that in the previous paper.  Based on this,  a linear time algorithm is devised to calculate the Steiner $k$-eccentricity of a vertex in a tree.  On the other hand,  the lower and upper bounds of the average Steiner $k$-eccentricity index of a tree on order $n$ are established based on a novel technique which is quite different from that in the previous paper but much easier to follow.
\end{abstract}


\medskip\noindent
\textbf{AMS Math.\ Subj.\ Class.\ (2020)}: 05C12, 05C05, 05C85

\section{Introduction}

In this paper we consider connected, simple, undirected graphs $G = (V(G), E(G))$. For basic graph notation and terminology we follow the book of West~\cite{west-2001}, while for algorithmic and computational terminology we use~\cite{Cormen2001,Garey1979}.

The standard distance $d_{G}(u,v)$ between vertices $u$ and $v$ in graph $G$ is the length of a shortest path between $u$ and $v$ in $G$. If $S\subseteq V(G)$, $|S|\ge 2$, then the {\em Steiner distance} $d_G(S)$ is the minimum size among all connected subgraphs of $G$ containing $S$, that is,
$$d_G(S) = \min\{|E(T)|: T\ {\rm is\ a\ subtree\ of}\ G, S\subseteq V(T)\}\,.$$
Note that if $S = \{u,v\}$, then $d_G(S) = d_{G}(u,v)$. If $k\ge 1$, then the \emph{Steiner $k$-eccentricity} of a vertex $v$ in graph $G$ is
$$\ec_{k}(v,G) = \max \{d_{G}(S):\ v\in S\subseteq V(G), |S|=k\}\,.$$
Note that, by definition, $\ec_{1}(v,G) = 0$. $S\subseteq V(G)$ is a {\em Steiner $k$-ecc $v$-set} if $|S|=k$, $v\in S$, and $d_{G}(S)=\ec_{k}(v,G)$. A corresponding minimum Steiner tree $T$ is called a {\em Steiner $k$-ecc $v$-tree} (corresponding to the $k$-set $S$). We will also shorty say that $T$ is a $MST(S,G)$. The average Steiner $k$-eccentricity of a graph $G$ is the mean value of all vertices' Steiner $k$-eccentricities in $G$, that is,
$${\rm aecc}_{k}(G) = \frac{1}{|V(G)|} \sum_{v\in V(G)} \ec_{k}(v,G)\,,$$
which is an extention of the average eccentricity of a graph~\cite{dankelmann2004, dankelmann2014}.

The Steiner tree problem is NP-hard on general graphs~\cite{Garey1979,Hwang1992}, but it can be solved in polynomial time on trees~ \cite{Beineke1996}. The Steiner distance on some special graph classes such as trees, joins, Corona products, threshold and product graphs, has been studied in~\cite{Anand2012,Chartrand1989,Gologranc,Mao2018,Wang2017}.  The average Steiner $k$-distance is closely related to the $k$-th Steiner Wiener index. Both of them were studied on trees, complete graphs, paths, cycles and complete bipartite graphs~\cite{Dankelmann1996,gutman-2017}. The average Steiner distance and the Steiner Wiener index were investigated in~\cite{Dankelmann1997,Li2016Wiener,lu}, while for some work on the Steiner diameter see~\cite{Mao2018,Wang2017}. The Steiner $k$-diameter was compared with the Steiner $k$-radius in~\cite{henning-1990, reiswig-2020}. Closely related invariants were also studied, for instance Steiner Gutman index~\cite{mao2018}, Steiner degree distance~\cite{Gutman2016}, Steiner hyper-Wiener index~\cite{Tratnik2019}, multi-center Wiener indices~\cite{Gutman2015}, and Steiner (revised) Szeged index~\cite{ghorbani}. We especiall point to the substantial survey~\cite{Mao2017Survey} on the Steiner distance and related results and to the recent investigation of isometric subgraphs for Steiner distance~\cite{weisauer-2020}.

Very recently, the Steiner $3$-eccentricity of trees was investigated in~\cite{Li2020}. A linear-time algorithm was developed to calculate the Steiner $3$-eccentricity of a vertex in a tree, and lower and upper bounds for the average Steiner $3$-eccentricity index on trees were derived. In this paper we extend these results to arbitrary $k\ge 2$. In the next section we propose  a linear algorithm to calculate the Steiner $k$-eccentricity of a vertex in a tree. In Section~\ref{bounds} we establish lower and upper bounds of the average Steiner $k$-eccentricity on trees. We conclude this paper by presenting several possibilities for future work.

\section{Steiner $k$-eccentricity of vertices in trees}
\label{algorithms}

The techniques from~\cite{Li2020} that enabled to calculate the Steiner $3$-eccentricity in a tree are not suitable for calculating the Steiner $k$-eccentricity of a vertex in a tree for arbitrary $k\ge 2$. In this section we establish new, stronger structural properties for the Steiner $k$-ecc $v$-tree for a vertex $v$ in a tree, and then apply them to devise a linear time algorithm to calculate the Steiner $k$-eccentricity of a vertex in a tree.

\subsection{Two key structural properties}

Before stating the two properties, let us introduce some notation and terminology on trees. A vertex of a tree of degree at least $3$ is a {\em branching vertex}. Let $L(T)$ denote the set of pendent vertices (leaves) of a tree $T$. If $u$ and $v$ are vertices of a tree $T$, then we will denote the (unique) $u.v$-path in $T$ by $P(u,v,T)$. Given a vertex $v\in V(T)$ and a leaf $u\in L(T)$, let $w$ be the nearest branching vertex to $u$ on $P(v,u,T)$. If there is no branching vertex on $P(v,u,T)$, we set $w=v$. Then we say that the sub-path $P(w,u,T)$ of $P(v,u,T)$ is a \emph{quasi-pendent path (with respect to $u$ and $v$)}.

In the rest we will use the following earlier lemma, also without explicitly mentioning it.

\begin{lemma} {\rm \cite[Lemmas 2.4, 2.5]{Li2020}}
\label{lem:earlier}
If $T$ is a tree and $v\in V(T)$, then the following holds.

(i) If $k > |L(T)|$, then every $k$-ecc $v$-set contains all the leaves of $T$. The same conclusion holds if $v$ is a leaf and $k = |L(T)|$.

(ii) If $2\le k \le |L(T)|$ and $S$ is a $k$-ecc $v$-set, then every vertex from $S\setminus \{v\}$ is a leaf of $T$.
\end{lemma}

For our first structural result, we need one more lemma.

\begin{lemma}\label{shared-path}
Let $k\ge 2$, let $v$ be a vertex of a tree $T$, let $T_{v}^{k}$ be a Steiner $k$-ecc $v$-tree, and let $T_{v}^{k-1}$ be a Steiner $(k-1)$-ecc $v$-tree. Then there exists a leaf $u\in L(T_{v}^{k})\setminus L(T_{v}^{k-1})$ such that the quasi-pendent path $P(w,u,T_{v}^{k})$ has no common edge with $T_{v}^{k-1}$.
\end{lemma}

\begin{proof}
If $k=2$, then $T_{v}^{1}$ is a tree on a single vertex $v$, hence the conclusion is clear. Assume in the rest that $k\ge 3$ and suppose on the contrary that every leaf $u\in L(T_{v}^{k})\setminus L(T_{v}^{k-1})$ satisfies that the quasi-pendant path $P(w,u,T_{v}^{k})$ has common edges with $T_{v}^{k-1}$. Then to every leaf $u\in L(T_{v}^{k})$ we can associate its private leaf of $L(T_{v}^{k-1})$. Hence the number of leaves in $T_{v}^{k-1}$ is not less than that in  $T_{v}^{k}$. This contradicts the fact (by Lemma~\ref{lem:earlier}) that the Steiner $(k-1)$-ecc $v$-set corresponding to $T_{v}^{k-1}$ has one less element than the Steiner $k$-ecc $v$-set corresponding to $T_{v}^{k}$.
\end{proof}

\begin{theorem}\label{k-ecc-tree-preserve}
Let $k\geq 2$, and let $v$ be a vertex of a tree $T$. Then every Steiner $k$-ecc $v$-tree contains some Steiner $(k-1)$-ecc $v$-tree.
\end{theorem}

\begin{proof}
The case $k=2$ is trivial, hence assume in the rest that $k\ge 3$. Let $T_{v}^{k}$ be a Steiner $k$-ecc $v$-tree and suppose on the contrary that it  contains no Steiner $(k-1)$-ecc $v$-tree. If $T_{v}^{k-1}$ is an arbitrary Steiner $(k-1)$-ecc $v$-tree, then, by Lemma~\ref{shared-path}, we may select a leaf $u$ from $T_{v}^{k}$ such that the quasi-pendant path $P(w,u,T_{v}^{k})$ does not have common edges with $T_{v}^{k-1}$.

Let $S_{v}^{k}$ be the Steiner $k$-ecc $v$-set corresponding to $T_{v}^{k}$ and set $S_{1}=S_{v}^{k}\setminus\{u\}$. Then $S_{1}$ is a ($k-1$)-set containing the vertex $v$. Moreover, the tree $T_{1}=T_{v}^{k}\setminus (P(w,u,T_{v}^{k})\setminus\{w\})$ is a $MST(S_1, T)$. By the assumption, the size of $T_{1}$ is strictly  less than that of $T_{v}^{k-1}$, that is,
\begin{equation}\label{T-1}
|E(T_{1})|<|E(T_{v}^{k-1})|\,.
\end{equation}

Let $S_{2}=S_{v}^{k-1}\cup \{u\}$, where $S_{v}^{k-1}$ is the Steiner ($k-1$)-ecc $v$-set corresponding to the tree $T_{v}^{k-1}$. Then $S_{2}$ is a $k$-set which contains the vertex $v$. Let $T_{2}$ be a $MST(S_{2},T)$. In the following we are going to show that the size of $T_{2}$ is larger than that of $T_{v}^{k}$.

Since the quasi-pendant path $P(w,u,T_{v}^{k})$ does not share any edge with $T_{v}^{k-1}$ and must be a sub-path of the quasi-pendant path $P(w',u,T_{2})$,  the size of $T_{2}$ satisfies
\begin{equation}\label{T-2}
|E(T_{2})|\geq |E(T_{v}^{k-1})|+|E(P(w,u,T_{v}^{k}))|\,.
\end{equation}
Combining~\eqref{T-1} and \eqref{T-2} we obtain that
\begin{align}\label{contra}
|E(T_{2})|  &\geq |E(T_{v}^{k-1})|+|E(P(w,u,T_{v}^{k}))|\nonumber\\
           &> |E(T_{1})|+|E(P(w,u,T_{v}^{k}))\nonumber\\
           &=|E(T_{v}^{k})|\,.
\end{align}
Hence $|E(T_{2})|>|E(T_{v}^{k})|$. Since $T_{2}$ is a minimum Steiner tree on a $k$-set containing $v$,~\eqref{contra} contradicts the fact that $T_{v}^{k}$ is a Steiner $k$-ecc $v$-tree.
\end{proof}

Theorem~\ref{k-ecc-tree-preserve} thus asserts that a Steiner $k$-ecc $v$-tree contains some Steiner $(k-1)$-ecc $v$-tree. The question now is, how to determine such a Steiner $(k-1)$-ecc $v$-tree. The message of the next result is that for our purposes, any Steiner $(k-1)$-ecc $v$-tree will do it. Before stating the theorem, we need some more notation. If $H$ is a subgraph of a graph $G$, and $v\in V(G)$, then the distance from $v$ to $H$ is $d_G(v,H) = \min\{d_{G}(v,u):u\in V(H)\}$. The eccentricity of $H$ in $G$ is $\ec_G(H) = \max\{d_G(v,H):v\in V(G)\}$.

\begin{theorem}\label{longest-path-no-matter}
Let $k\ge 1$, and let $v$ be a vertex of a tree $T$. If $T_{1}$ and $T_{2}$ are Steiner $k$-ecc $v$-trees of $T$,  then  $\ec_T(T_{1})= \ec_T(T_{2})$.
\end{theorem}

\begin{proof}
There is nothing to be proved if $T_1 = T_2$. Hence assume in the rest that  $T_{1}$ and $T_{2}$ are different Steiner $k$-ecc $v$-trees of $T$. If $k=1$, then a (unique) Steiner $1$-ecc $v$-tree is induced by the vertex $v$ itself. Since all longest paths starting from $v$ have the same length, the assertion of the theorem is clear for $k=1$. Hence we may also assume in the rest of the proof that $k\ge 2$.

Let $P_{1}$ and $P_{2}$  be longest paths from vertices of $V(T)$ to trees $T_{1}$ and $T_{2}$, respectively. Let
$u_{1}$ and $u_{2}$ be the two endpoints of $P_{1}$ with $u_{1}\in V(T_1)$, and let $w_{1}$ and $w_{2}$ be the two endpoints of $P_{2}$ with $w_{1}\in V(T_{2})$. Set $T_{0}=T_{1}\cap T_{2}$. To prove the theorem it suffices to  prove that  $u_{1}\in V(T_{0})$ and $w_{1}\in V(T_{0})$. By symmetry, it suffices to prove the first assertion, that is, $u_{1}\in V(T_{0})$.

Suppose on the contrary that $u_{1}\in V(T_{1})\setminus V(T_{0})$. Let $s$ be a leaf of $T_{1}$ such that $u_{1}$ is on the path $P(v,s, T_{1})$. Then there must be a vertex $w_{0}\in V(T_{0})$ and a leaf $t$ of $T_{2}$ such that $E(P(w_{0},s,T_{1}))\cap E(P(w_{0},t,T_{2}))=\emptyset$, see Fig. \ref{proof-lemma-2-3}. Note that $w_{0}$ may be the vertex $v$.

\begin{figure}[ht!]
\vspace{0.5cm}
\begin{center}
\epsfig{file=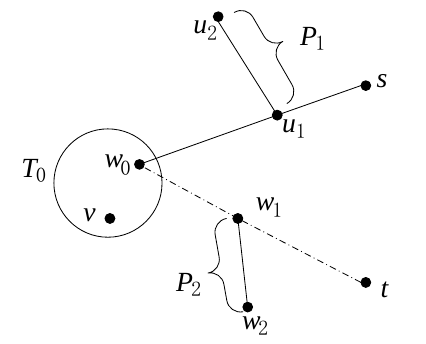, scale=1.3}
\end{center}
\vspace{-0.5cm}
\caption{The configuration of the vertices $w_{0}$, $u_{1}$, $u_2$, $w_{1}$, $w_{2}$, $s$ and $t$.}
\label{proof-lemma-2-3}
\end{figure}

We claim that $V(P_{1})\cap V(T_{2})=\emptyset$. Otherwise, let $x\in V(T_{2})\cap V(P_{1})$. Then the path $E(P(x,v,T_{1}))\setminus E(P(x,v,T_{2}))\not = \emptyset$, since $E(P(w_{0},u_{1},T_{1}))\not= \emptyset$. So the two paths $P(x,v,T_{1})$ and $P(x,v,T_{2})$ form a cycle in the original graph $T$. This contradicts to the fact that $T$ is a tree. In the same way, we obtain that  $V(P_{2})\cap V(T_{2})=\emptyset$.

Since $|E(P_{1})| = d_T(u_{2},T_{1}) = \ec_T(T_{1})$ and $|E(P(w_{0},t,T_{2}))| = d_T(t,T_{1})$, we have
\begin{equation}\label{P-1-length}
|E(P_{1})|\geq |E(P(w_{0},t,T_{2}))|\,.
\end{equation}
Moreover, since we have assumed that $u_{1}\in V(T_{1})\setminus V(T_{0})$, we infer that $|E(P(u_{1},w_{0},T))|>0$. Together with~\eqref{P-1-length} this yields
    \begin{align}\label{u-2-to-T-2}
    |E(P(u_{2},w_{0},T))|&=|E(P_{1})|+|E(P(u_{1},w_{0},T))|\nonumber\\
                         &\geq |E(P(w_{0},t,T_{2}))|+|E(P(u_{1},w_{0},T))|\nonumber \\
                         &>|E(P(w_{0},t,T_{2}))|\,.
    \end{align}

    Now we pay attention to the tree $T_{2}$.  Let $S$ be the Steiner $k$-ecc $v$-set corresponding to the tree $T_{2}$. Let $S'=S\setminus\{t\}\cup \{u_{2}\}$. Then $S'$ is a $k$-set containing the vertex $v$. In the following, we will establish a contradition that the tree $T_{2}'=MST(S',T)$ has more edges than the tree $T_{2}$. Recall that $T_{2}$ is a Steiner $k$-ecc $v$-tree.

Let $P(w,t,T_{2})$ be the quasi-pendant path with respect to $v$ in $T_{2}$ and distinguish the following cases.

\medskip\noindent
\textbf{Case 1}: $w\in V(P(w_{0},t, T_{2}))\setminus \{w_{0}\}$. \\
In this case the tree $T_{2}'=MST(S',T)$ can be represented as $T_{2}'=T_{2}\setminus P(w,t,T_{2})\cup P(w_{0},u_{2},T)$. Since the path $P(w,t,T_{2})$ is a sub-path of $P(w_{0},t,T_{2})$, $|E(P(w_{0},t,T_{2}))|\ge |E(P(w,t,T_{2}))|$ holds. Combining this fact with~\eqref{u-2-to-T-2} we have:
    \begin{align*}
    |E(T_{2}')| &= |E(T_{2})|-|E(P(w,t,T_{2}))|+|E(P(w_{0},u_{2},T))|\\
                &\geq |E(T_{2})|-|E(P(w_{0},t,T_{2}))|+|E(P(w_{0},u_{2},T))|\\
                &>|E(T_{2})|\,.
     \end{align*}

\medskip\noindent
\textbf{Case 2}: $w\in V(T_{0})$.\\
Now the tree $T_{2}'=MST(S',T)$ can be represented as $T_{2}' = T_{2}\setminus P(w,t,T_{2})\cup P(w,u_{2},T)$.  Recall that the path $P(w,t,T_{2})$ is composed of two sub-paths which are $P(w,w_{0},T_{2})$ and $P(w_{0},t,T_{2})$ respectively. And $P(w,u_{2},T)$ is also composed of two sub-paths which are $P(w,w_{0},T_{2})$ and $P(u_{2},w_{0},T)$. By~\eqref{u-2-to-T-2} we can estimate as follows:
    \begin{align*}
    |E(T_{2}')|&=|E(T_{2})|-|E(P(w,t,T_{2}))|+|E(P(w,u_{2},T))|\\
              &=|E(T_{2})|-(|E(P(w,w_{0},T_{2}))|+|E(P(w_{0},t,T_{2}))|)\\
              & +(|E(P(w,w_{0},T_{2}))|+|E(P(u_{2},w_{0},T))|)\\
              &=|E(T_{2})|-|E(P(w_{0},t,T_{2}))|+|E(P(u_{2},w_{0},T))|\\
              &>|E(T_{2})|\,.
    \end{align*}
In both cases we have thus proved that  $|E(T_{2}')| > |E(T_{2})|$, a contradiction to the fact that $T_{2}$ is a Steiner $k$-ecc $v$-tree.
\end{proof}

\subsection{A linear time algorithm}

By Theorems~\ref{k-ecc-tree-preserve} and~\ref{longest-path-no-matter}, the problem to calculate the Steiner $k$-eccentricity of a  given vertex of a tree can be reduced to recursively finding a longest path starting at a given vertex. This is formally done in Algorithm~\ref{k-ecc}.

\begin{algorithm}[H]\label{k-ecc}
\caption{k-ECC($v$, $T$, $k$)}%算法名字
\LinesNumbered %要求显示行号
\KwIn{A vertex $v$, a tree $T$, and an integer $k\geq 2$}%输入参数
\KwOut{The Steiner $k$-eccentricity of $v$ in $T$}%输出
%some description\; %\;用于换行
\If {the number of leaves is less than $k$}{\label{non-cut-start}
    \Return $|V(T)|-1$\;
}\label{non-cut-end}
\Else{\label{start}
    $ecc=0$\;
    \For{$i=1$ {\bf to} $k-1$}{
        Longest\_Path($v$, $T$, path)\;\label{find-lingest-shortest-path}
        $ecc=ecc+|E(P)|$\;
        Path\_Shrinking($v$, $T$, path)\; \label{shrinking}
        }
    \Return $ecc$
    }\label{end}
\end{algorithm}

\bigskip
To explain Steps 1-3 of Algorithm~\ref{k-ecc}, we state the following lemma.

\begin{lemma}\label{trivial-case}
Let $k\ge 3$ and let $v$ be a vertex of a tree $T$. If $|L(T)| < k$, then the Steiner $k$-ecc $v$-tree is the entire tree $T$.
\end{lemma}

\begin{proof}
The cardinality of the set $S=\{v\}\cup L(T)$ is at most $k$, since $|L(T)|<k$. Moreover, the $MST(S,T)$ is the entire tree $T$. Hence the Steiner $k$-ecc $v$-tree is the entire $T$.
\end{proof}

Steps \ref{start}-\ref{end} form the recursive reduction which consists of finding $k-1$ times a longest path starting at a vertex. In Step~7 we use the depth-first search (DFS) algorithm~\cite{Cormen2001} to find a longest path starting at a given vertex, tte details are present in Algorithm~\ref{alg-longest-path-at-a-vertex}. Step~9 shrinks the path obtained in Step~7 into a single vertex for the purpose of the next loop, the details are presented in Algorithm~\ref{alg-shrink-graph}. Algorithms~\ref{alg-longest-path-at-a-vertex} and~ \ref{alg-shrink-graph} are borrowed from~\cite{Li2020} where one can find additional details on them. For the statement of these algorithms we recall that if $v$ is a vertex of a graph $G$, then the set of its neighbours is denoted by $N_{G}(v)$.

\begin{algorithm}[H]\label{alg-longest-path-at-a-vertex}
\caption{Longest\_Path($v$, $T$, path)}%算法名字
\LinesNumbered %要求显示行号
\KwIn{A vertex $v$, a tree $T$ rooted at $v$, and an array named \emph{path} to store a longest path starting at $v$}%输入参数
\KwOut{the length of a longest path starting at $v$}%输出
%some description\; %\;用于换行
max=0; temp=max;\\
\For{each vertex $u\in N_{T}(v)$ which has not been visited till now}{
    %done=1\;
　　temp=Longest\_Path($u$, $T$, path)\;
　　\If{temp$>$max}{
        path[$v$]=$u$\;
　　　　max=temp;
　　}
}
%\If{done==0}{
%    path[$v$]=$\emptyset$\;
%}
\Return max+1\;
\iffalse
\While{not at end of this document}{
　　if and else\;
　　\eIf{condition}{
　　　　1\;
　　}{
　　　　2\;
　　}
}
\ForEach{condition}{
　　\If{condition}{
　　　　1\;
　　}
}
\fi
\end{algorithm}

\begin{algorithm}[H]\label{alg-shrink-graph}
\caption{Path\_Shrinking($v$, $T$, path)}%算法名字
\LinesNumbered %要求显示行号
\KwIn{A tree $T$, a vertex $v$, and an array named \emph{path} to store a longest path starting at $v$}% 输入参数
\KwOut{A new tree obtained by shrinking the longest path into the single vertex $v$}%输出
%some description\; %\;用于换行
w=v\;
\While{\emph{path[w]}$\not =\emptyset$}{
    \For{each vertex $x\in N_{T}(w)$}{
        remove the edge $(w,x)$ from $T$\;
        add a new edge between $x$ and $v$ in $T$\;
    }
    w=path[w];
}
\end{algorithm}

\begin{theorem}\label{crooectness-and-complexity}
Algorithm \ref{k-ecc} computes the Steiner $k$-eccentricity of a vertex in a tree and can be implemented to run in $O(k(n+m))$ time, where $n$ and $m$ are the order and the size of the tree, respectively.
\end{theorem}

\begin{proof}
The correctness of Algorithm \ref{k-ecc} is ensured by Theorems~\ref{k-ecc-tree-preserve} and~\ref{longest-path-no-matter}.

By Lemma \ref{trivial-case}, the Steiner $k$-eccentricity of a vertex in a tree is equal to the size of the tree if its number of leaves is less than $k$. There is a linear-time algorithm to find all leaves of a tree by the depth-first search (DFS) algorithm \cite{Cormen2001}. Hence Steps \ref{non-cut-start}-\ref{non-cut-end} can be implemented in $O(n+m)$ time. Similarly, each loop in Steps~6-9\ can be implemented in $O(n+m)$ time, thus all loops require $O(k(n+m))$ time.
\end{proof}

To conlcude the section we again point out that the structural properties to support the algorithm(s) from~\cite{Li2020}  only ensure calculation of the Steiner $3$-eccentricity. Hence we need to develop a new  approach that works for general $k$.

\section{Upper and lower bounds}\label{bounds}

In this section we establish an upper and a lower bound on the average Steiner $k$-eccentricity index of a tree for $k\geq 3$. These bounds were earlier proved in~\cite{Li2020} in the special case $k=3$. It is appealing that to obtain the bound for the general case, the proof idea is quite different and significantly simpler that the one in~\cite{Li2020}. For the new approach, the following construction is essential.

\medskip\noindent
\textbf{$\pi$-transformation}:
Let $T$ be a tree and let $P = P(u,v,T)$ be a path with at least one edge, such that every internal vertex of $P$ is of degree $2$ in $T$. Let $X$ be the maximal subtree containing $u$ in the tree $T\setminus E(P)$, and $Y$ be the  maximal subtree containing $v$ in the graph $T\setminus E(P)$. We may without loss of generality assume that $\ec_T(u,X)\leq \ec_T(v,Y)$. Then the $\pi$-\emph{transformation} $\pi(T)$ of  $T$ is defined as $T' = \pi(T) = T\setminus\{(u,w):w\in N_{X}(u)\}\cup\{(v,w):w\in N_{X}(u)\}$. The inverse transformation is is $T=\pi^{-1}(T')=T'\setminus\{(v,w):w\in N_{X}(v)\}\cup\{(u,w):w\in N_{X}(v)\}$. See Fig.~\ref{pi-trans}.

\begin{figure}[htbp]
\vspace{0.5cm}
\begin{center}
\epsfig{file=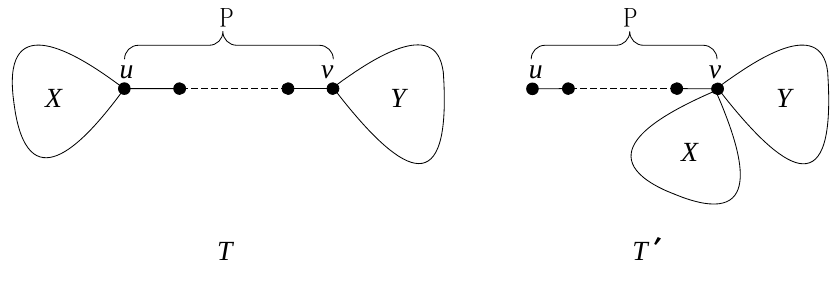, scale=1.3}
\end{center}
\vspace{-0.5cm}
\caption{$T'=\pi(T)$ and $T=\pi^{-1}(T')$}
\label{pi-trans}
\end{figure}

\begin{lemma}\label{vertex-on-P}
Let $T$, $P$, $u$, $v$, $X$, $Y$, and $T'$ be as in the definition of the $\pi$-transformation. If $w\in V(P)\cup V(X)$, then in $T'$ there exists a Steiner $k$-ecc $w$-set $S$ such that $S\cap (V(Y)\setminus\{v\})\not =\emptyset$.
\end{lemma}

\begin{proof}
Let $S'$ be a Steiner $k$-ecc $w$-set in $T'$ such that $S\cap (V(Y)\setminus\{v\}) =\emptyset$, and set  $Q=S'\setminus \{w\}$. Since $k\ge 3$, the cardinality of $Q$ is at least two. Let $v'\in V(Y)$ such that the distance between $v$ and $v'$ is $\ec_{T'}(v,Y)$. Consider the following two cases.

\medskip\noindent
\textbf{Case 1}: $Q\cap V(X)=\emptyset$. \\
In this case the vertices of $Q$ are all in $P$. Let $w'\in Q$ be the nearest vertex to $v$. Construct a new vertex set $S''=(S'\setminus\{w'\}) \cup \{v'\}$.

\medskip\noindent
\textbf{Case 2}: $Q\cap V(X)\not =\emptyset$. \\
Let $w'\in Q\cap V(X)$. Construct a new vertex set $S''= (S'\setminus\{w'\}) \cup \{v'\}$.

In each of the two cases, the size of $MST(S'',T')$ is not less than the size of  $MST(S',T')$, hence the assertion.
\end{proof}

\begin{lemma}\label{pi-not-increase}
Under the notation of Lemma~\ref{vertex-on-P}, ${\rm aecc}_{k}(T)\geq {\rm aecc}_{k}(T')$.
\end{lemma}

\begin{proof}
If $w$ is a vertex in $V(Y)\setminus \{v\}$, then for any Steiner $k$-ecc $w$-set $S'$ in $T'$, the size of a minimum Steiner tree on $S'$ in graph $T$ is not less than that in $T'$. So the Steiner $k$-eccentricity of every vertex $w\in V(Y)$ in $T$ is not less than that in $T'$.

If $w$ is a vertex in $V(P)\cup V(X)$, then by Lemma \ref{vertex-on-P}, there exists a Steiner $k$-ecc $w$-set $S'$ in $T'$, such that $S'\cap (V(Y)\setminus \{v\})\not =\emptyset$. The size of a minimum Steiner tree on $S'$ in $T$ is not less than that in $T'$. Therefore the Steiner $k$-eccentricity of every vertex $w\in V(P)\cup V(X)$ in $T$ is not less than that in $T'$.

In any case, the Steiner $k$-eccentricity of every vertex $v\in V(T')$ is not larger than that in $T$. As the average Steiner $k$-eccentricity index is the mean value of all vertices' Steiner $k$-eccentricities, the average Steiner $k$-eccentricity of $T'$ is not large than that of $T$.
\end{proof}

If the order of a tree $T$ is not larger than $k$, then a Steiner $k$-ecc $v$-set contains all vertices of $T$ for every $v\in V(T)$. Then every Steiner $k$-ecc $v$-tree is the entire tree $T$ for every vertex $v$. So for a given $k\ge 3$, we just consider the trees where the order of each is more than $k$.

\begin{theorem}\label{bounds-on-trees}
If $k\geq 3$ is an integer, and $T$ a tree on order $n>k$, then
    \[ k-\frac{1}{n}\leq {\rm aecc}_{k}(T)\leq n-1.\]
Moreover, the star $S_n$ attains the lower bound, and the path $P_n$ attains the upper bound.
\end{theorem}

\begin{proof}
Repeatedly applying the $\pi$-transformation on $T$ until it is possible, we obtain the star $S_n$. On the other hand, repeatedly applying the $\pi^{-1}$ transformation on $T$ until it is possible,  we obtain the path $P_n$. By Lemma \ref{pi-not-increase}, the $\pi$-transformation does not increase the average Steiner $k$-eccentricity of $T$. Hence the star $S_n$ attains the minimum Steiner $k$-eccentricity, and the path $P_n$ attains the maximum Steiner $k$-eccentricity. Finally, we obtain ${\rm aecc}_{k}(S_n)=k-\frac{1}{n}$ and ${\rm aecc}_{k}(P_n)=n-1$ by straightforward computation.
\end{proof}

In Fig.~\ref{extremal-graphs} an example is given in which the process of constructing extremal graphs, that is, a start and a path, by means of  the $\pi$-transformation and the $\pi^{-1}$-transformation.

\begin{figure}[htbp]
%\vspace{-1.5cm}
\begin{center}
\epsfig{file=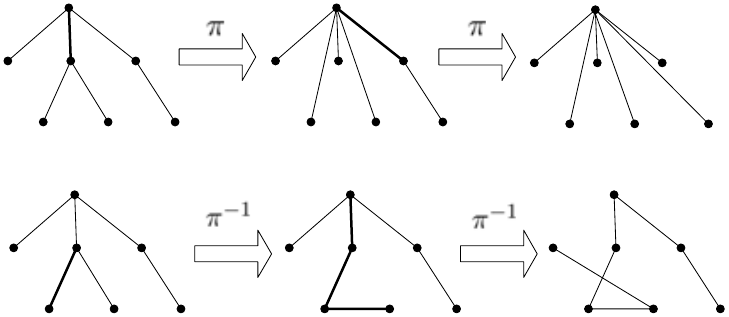, scale=1.3}
\end{center}
\vspace{-0.5cm}
\caption{Constructing extremal graphs using the $\pi$ transformation and the $\pi^{-1}$ transformation. Bold edges denote the paths defined in the transformations.}
\label{extremal-graphs}
\end{figure}

In~\cite{Ilic2012} the average Steiner $2$-eccentricity of trees was investigated. For the sake of our final result,  we recall the following result.

\begin{lemma} {\rm (\cite{Ilic2012})}
\label{2-ecc}
Let $T$ be a tree of order $n$. Then ${\rm aecc}_{2}(S_{n})\leq {\rm aecc}_{2}(T)\leq {\rm aecc}_{2}(P_{n})$. The left equality holds if and only if $T\cong S_{n}$, while the right equality holds if and only if $T\cong P_{n}$.
\end{lemma}

Combining Theorem~\ref{bounds-on-trees} with Lemma~\ref{2-ecc}, we have the following result.

\begin{corollary}\label{combination}
If $k\geq 2$ is an integer, then  $S_{n}$ (resp.\  $P_n$) attains the minimum (resp.\ the maximum) average Steiner $k$-eccentricity in the class of trees.
\end{corollary}

\section{Conclusion}\label{conclusion}

In this paper we have derived a linear-time algorithm to calculate the Steiner $k$-eccentricity of a vertex in a tree, and established lower and upper bounds for the average Steiner $k$-eccentricity of a tree. These results extend  those from~\cite{Li2020} for the case $k=3$. It remains open to determine the extremal graphs for the average Steiner $k$-eccentricity index on trees for $k\geq 2$. Moreover, the general problem to compute the Steiner $k$-eccentricity of a general graph is widely open, in particular, it is not known whether it is NP-hard.

\end{document}